\documentclass[a4paper,11pt,reqno]{amsart}
\usepackage[utf8]{inputenc}
\usepackage{amsmath}
\usepackage{amssymb}
\usepackage{amsthm}
\usepackage{mathtools}
\usepackage{hyperref}
\usepackage{color}
\usepackage{enumitem}
\usepackage{a4wide}
\usepackage[all]{xy}

\newcommand{\frakg}{\mathfrak{g}}

\newcommand{\frakk}{\mathfrak{k}}

\newcommand{\frakp}{\mathfrak{p}}

\newcommand{\frakt}{\mathfrak{t}}

\newcommand{\CC}{\mathbb{C}}

\newcommand{\NN}{\mathbb{N}}

\newcommand{\calF}{\mathcal{F}}
\newcommand{\calH}{\mathcal{H}}

\newcommand{\calM}{\mathcal{M}}

\newcommand{\calU}{\mathcal{U}}

\DeclareMathOperator{\Hom}{Hom}

\DeclareMathOperator{\id}{id}

\theoremstyle{plain}
\newtheorem{theorem}{Theorem}
\newtheorem*{theorem*}{Theorem}

\newtheorem{lemma}[theorem]{Lemma}

\newtheorem*{fact*}{Fact}

\newtheorem{thmalph}{Theorem}
\newtheorem{coralph}[thmalph]{Corollary}

\theoremstyle{definition}

\newtheorem{remark}[theorem]{Remark}

\title[Direct integral decomposition in branching laws for real reductive groups]{On the direct integral decomposition in branching laws for real reductive groups}

\author{Jan Frahm}
\address{Department of Mathematics, Aarhus University, Ny Munkegade 118, 8000 Aarhus C, Denmark}
\email{frahm@math.au.dk}

\begin{document}



\maketitle

\begin{abstract}
The restriction of an irreducible unitary representation $\pi$ of a real reductive group $G$ to a reductive subgroup $H$ decomposes into a direct integral of irreducible unitary representations $\tau$ of $H$ with multiplicities $m(\pi,\tau)\in\NN\cup\{\infty\}$. We show that on the smooth vectors of $\pi$, the direct integral is pointwise defined. This implies that $m(\pi,\tau)$ is bounded above by the dimension of the space $\Hom_H(\pi^\infty|_H,\tau^\infty)$ of intertwining operators between the smooth vectors, also called \emph{symmetry breaking operators}, and provides a precise relation between these two concepts of multiplicity.
\end{abstract}

\section*{Introduction}

Let $G$ be a real reductive group and $H\subseteq G$ a reductive subgroup. Then, for any irreducible unitary representation $(\pi,\calH_\pi)$ of $G$, its restriction to $H$ decomposes into a direct integral of irreducible unitary representations $(\tau,\calH_\tau)$ of $H$, i.e. there exists an $H$-equivariant unitary isomorphism
\begin{equation}
	T:\calH_\pi \to \int^\oplus_{\widehat{H}}\calH_\tau\otimes\calM_{\pi,\tau}\,d\mu_\pi(\tau),\label{eq:DirectIntegralDecomposition}
\end{equation}
where $\calM_{\pi,\tau}$ is a family of Hilbert spaces (called \emph{multiplicity spaces}) and $\mu_\pi$ a Borel measure on the unitary dual $\widehat{H}$ of $H$. Here, $H$ acts only on the first factor of $\calH_\tau\otimes\calM_{\pi,\tau}$, so
$$ m(\pi,\tau)=\dim\calM_{\pi,\tau}\in\NN\cup\{\infty\} $$
is the \emph{multiplicity of $\tau$ in $\pi$}. Note that the function $\tau\mapsto m(\pi,\tau)$ is unique up to a set of measure zero, and in this sense the direct integral decomposition is unique.

The map $T$ is in general not \emph{pointwise defined}, i.e. there do not exist continuous linear maps $A_{\pi,\tau}:\calH_\pi\to\calH_\tau\otimes\calM_{\pi,\tau}$ such that $T(v)_\tau=A_{\pi,\tau}(v)$ for $\mu_\pi$-almost every $\tau\in\widehat{H}$. In fact, the existence of a non-zero $H$-equivariant continuous linear map $\calH_\pi\to\calH_\tau$ implies that $\tau$ occurs discretely inside $\pi|_H$, i.e. $\mu_\pi(\{\tau\})>0$. To also capture the continuous part of the decomposition, one has to restrict to a subspace of $\calH_\pi$. We show that the dense subspace $\calH_\pi^\infty$ of smooth vectors is sufficient for this purpose:

\begin{thmalph}\label{thm:Main}
	The restriction of $T$ to the smooth vectors $\calH_\pi^\infty$ is pointwise defined, i.e. for every $\tau\in\widehat{H}$ there exists an $H$-equivariant continuous linear map $A_{\pi,\tau}^\infty:\calH_\pi^\infty\to\calH_\tau\otimes\calM_{\pi,\tau}$ such that
	$$ T(v)_\tau = A_{\pi,\tau}^\infty(v) \qquad \mbox{for $\mu_\pi$-almost every $\tau\in\widehat{H}$.} $$
\end{thmalph}

We remark that every $H$-equivariant continuous linear map $\calH_\pi^\infty\to\calH_\tau$ automatically maps into the smooth vectors $\calH_\tau^\infty$ of $\tau$, i.e. is contained in $\Hom_H(\calH_\pi^\infty,\calH_\tau^\infty)$. Such operators are also called \emph{symmetry breaking operators} (see Kobayashi~\cite{Kob15}). The space of symmetry breaking operators has been studied intensively in connection with finite multiplicity/multiplicity one statements (see e.g. \cite{KO13,SZ12}) and branching laws for specific pairs of groups $(G,H)$ (see e.g. \cite{CO11,FW20,KS15,Moe17}). Theorem~\ref{thm:Main} shows that the direct integral decomposition \eqref{eq:DirectIntegralDecomposition} of an irreducible unitary representation can always be constructed in terms of symmetry breaking operators. This statement might have been common knowledge, but we could not find a reference in the literature.

Theorem~\ref{thm:Main} implies an upper bound for the multiplicity $m(\pi,\tau)$ in terms of the dimension of the space $\Hom_H(\pi^\infty|_H,\tau^\infty)=\Hom_H(\calH_\pi^\infty,\calH_\tau^\infty)$ of symmetry breaking operators between the smooth vectors of $\pi$ and $\tau$:

\begin{coralph}\label{cor:Main}
	For $\mu_\pi$-almost every $\tau\in\widehat{H}$:
	$$ m(\pi,\tau)\leq\dim\Hom_H(\pi^\infty|_H,\tau^\infty). $$
\end{coralph}

We remark that, although this upper bound is attained for some representations $\tau$, the right hand side might be strictly larger in many cases. This means that not every non-trivial symmetry breaking operator contributes to the decomposition of the unitary representation.\\

\textbf{Acknowledgments.} We thank an anonymous referee for pointing out the reference \cite{Li18}. The author was supported by a research grant from the Villum Foundation (Grant No. 00025373).

\section{Direct integrals of Hilbert spaces}

We briefly recall the construction of direct integrals of Hilbert spaces following the exposition in \cite[Section 8]{KS18}. All Hilbert spaces are assumed to be separable.

Let $(\calH_\lambda)_{\lambda\in\Lambda}$ be a family of Hilbert spaces indexed by a second-countable topological space $\Lambda$ and let $\mu$ be a $\sigma$-finite Borel measure on $\Lambda$. Denote by $\langle\cdot,\cdot\rangle_\lambda$ the inner product on $\calH_\lambda$. We identify elements $s\in\prod_{\lambda\in\Lambda}\calH_\lambda$ with sections $s:\Lambda\to\bigsqcup_{\lambda\in\Lambda}\calH_\lambda$ satisfying $s(\lambda)\in\calH_\lambda$ for every $\lambda\in\Lambda$. Suppose we are given a subspace $\calF\subseteq\prod_{\lambda\in\Lambda}\calH_\lambda$, called the space of \emph{measurable sections}, satisfying:
\begin{itemize}
	\item For all $s,t\in\calF$ the map $\Lambda\to\CC,\,\lambda\mapsto\langle s(\lambda),t(\lambda)\rangle_\lambda$ is measurable.
	\item If $s\in\prod_{\lambda\in\Lambda}\calH_\lambda$ is such that $\lambda\mapsto\langle s(\lambda),t(\lambda)\rangle_\lambda$ is measurable for all $t\in\calF$, then $s\in\calF$.
	\item There exists a countable subset $(s_n)_{n\in\NN}\subseteq\calF$ such that $\{s_n(\lambda):n\in\NN\}$ spans a dense subspace of $\calH_\lambda$ for every $\lambda\in\Lambda$.
\end{itemize}
The family $(\calH_\lambda)_{\lambda\in\Lambda}$ together with the measure $\mu$ and the subspace $\calF$ of measurable sections is called a \emph{measurable family of Hilbert spaces}. The direct integral of such a family is defined as the Hilbert space
$$ \int^\oplus_\Lambda\calH_\lambda\,d\mu(\lambda) = \left.\left\{s\in\calF:\int_\Lambda\langle s(\lambda),s(\lambda)\rangle_\lambda<\infty\right\}\right/\sim $$
of square integrable sections modulo the subspace of sections which are zero almost everywhere. This Hilbert space carries the obvious inner product.

A continuous linear map $T:E\to\int^\oplus_\Lambda\calH_\lambda\,d\mu(\lambda)$ from a topological vector space $E$ into a direct integral is said to be \emph{pointwise defined} if there exists a continuous linear map $T_\lambda:E\to\calH_\lambda$ for every $\lambda\in\Lambda$ such that for every $v\in E$:
$$ T(v)(\lambda)=T_\lambda(v) \qquad \mbox{for almost every }\lambda\in\Lambda. $$

\begin{theorem}[{Gelfand--Kostyuchenko, see e.g. \cite[Theorem 1.5]{Ber88}}]\label{thm:GKThm}
	Every Hilbert--Schmidt operator $T:\calH\to\int^\oplus_\Lambda\calH_\lambda\,d\mu(\lambda)$ is pointwise defined.
\end{theorem}

The following statements are known by \cite[Lemma 1.3]{Ber88}. We include a short proof for convenience.

\begin{lemma}\label{lem:DenseImageAndEquivariance}
	Assume that $T:\calH\to\int^\oplus_\Lambda\calH_\lambda\,d\mu(\lambda)$ is pointwise defined
	\begin{enumerate}
		\item\label{lem:DenseImageAndEquivariance1} If $T$ has dense image, then almost every $T_\lambda:\calH\to\calH_\lambda$ has dense image.
		\item\label{lem:DenseImageAndEquivariance2} If $T$ is equivariant with respect to continuous representations of a Lie group $H$ on $\calH$ and each $\calH_\lambda$, then almost every $T_\lambda$ is equivariant.
	\end{enumerate}
\end{lemma}

\begin{proof}
	For part \eqref{lem:DenseImageAndEquivariance1}, let $(s_n)\subseteq\int^\oplus_\Lambda\calH_\lambda\,d\mu(\lambda)$ as in the definition of the direct integral. For each $n$ we let
	$$ \Lambda_n = \{\lambda\in\Lambda:s_n(\lambda)\notin\overline{T_\lambda(\calH)}\}. $$
	We first show that every $\Lambda_n$ has measure zero. Assume to the contrary that some $\Lambda_n$ has positive measure. Then
	\begin{align*}
		\|s_n-T(v)\|^2 &= \int_\Lambda\|s_n(\lambda)-T_\lambda(v)\|^2\,d\mu(\lambda) \geq \int_{\Lambda_n} \|s_n(\lambda)-T_\lambda(v)\|^2\,d\mu(\lambda)\\
		&= \int_{\Lambda_n} \|s_n(\lambda)_\perp\|^2+\|s_n(\lambda)_\|-T_\lambda(v)\|^2\,d\mu(\lambda)\ \geq \int_{\Lambda_n} \|s_n(\lambda)_\perp\|^2\,d\mu(\lambda),
	\end{align*}
	where $s_n(\lambda)_\|$ resp. $s_n(\lambda)_\perp$ denotes the orthogonal projection of $s_n(\lambda)$ to $\overline{T_\lambda(\calH)}$ resp. $T_\lambda(\calH)^\perp$. The latter integral is positive since $s_n(\lambda)_\perp\neq0$ on $\Lambda_n$ which is of positive measure. This shows that $\|s_n-T(v)\|\geq c>0$ for all $v\in\calH$, contradicting the fact that $T$ has dense image.\\
	We have shown that $\Lambda_n$ is of measure zero for every $n$, hence the countable union $\bigcup_n\Lambda_n$ has measure zero. This implies for almost every $\lambda$ that $s_n(\lambda)\in\overline{T_\lambda(\calH)}$ for all $n$, so by the assumption that $(s_n(\lambda))_n$ spans a dense subspace of $\calH_\lambda$ for every $\lambda$, we obtain that $\overline{T_\lambda(\calH)}=\calH_\lambda$ for almost every $\lambda$.
	
	To show \eqref{lem:DenseImageAndEquivariance2}, denote by $\pi$ the representation of $H$ on $\calH$ and by $\tau_\lambda$ the representation of $H$ on $\calH_\lambda$, $\lambda\in\Lambda$. For fixed $h\in H$ and $v\in\calH$, the $H$-equivariance of $T$ implies that
	$$ [T_\lambda\circ\pi(h)](v) = [\tau_\lambda(h)\circ T_\lambda](v) \qquad \mbox{for almost every }\lambda\in\Lambda. $$
	Letting $h$ resp. $v$ run through a countable dense subset of $H$ resp. $\calH$ shows the claim.
\end{proof}

\section{Sobolev norms on unitary representations}

Let $G$ be a real reductive group, $\frakg$ its Lie algebra and $\frakg=\frakk\oplus\frakp$ a Cartan decomposition. Denote by $\Delta_\frakk\in\calU(\frakk)$ the Casimir element of $\frakk$ with respect to an invariant inner product on $\frakk$. For an irreducible unitary representation $(\pi,\calH)$ of $G$ we write $(\pi^\infty,\calH^\infty)$  for the subrepresentation on the Fréchet space of smooth vectors. For every $N\in\NN$,
$$ \|v\|_N^2 = \sum_{j=0}^N \|(d\pi(\Delta_\frakk^j)v\|^2 \qquad (v\in\calH^\infty) $$
defines a continuous norm $\|\cdot\|_N$ on $\calH^\infty$ which dominates the Hilbert space norm $\|\cdot\|$ of $\calH$. Therefore, the completion $\calH^N$ of $\calH^\infty$ with respect to the norm $\|\cdot\|_N$ naturally embeds into $\calH$ and yields a scale of Hilbert spaces
$$ \calH^\infty\subseteq\ldots\subseteq\calH^{N+1}\subseteq\calH^N\subseteq\ldots\subseteq\calH^0=\calH. $$

\begin{lemma}\label{lem:SobolevEmbedding}
	For $4N>\dim\frakk$ the embedding $\calH^N\hookrightarrow\calH$ is Hilbert--Schmidt.
\end{lemma}

\begin{proof}
	Note that we may assume $G$ to be connected by decomposing $\calH$ into the direct sum of finitely many irreducible representations of the identity component $G_0$ of $G$.\\
	Next, decompose $\calH=\bigoplus_{\sigma\in\widehat{K}}\calH[\sigma]$ into $K$-isotypic components. Then $d\pi(\Delta)|_{\calH[\sigma]}$ is a scalar multiple of the identity. To describe the scalar, we identify an irreducible representation $\sigma$ of $K$ with its highest weight in $i\frakt^*$ with respect to a maximal torus $\frakt\subseteq\frakk$ and a system of positive roots. Then
	$$ d\pi(\Delta_\frakk)|_{\calH[\sigma]} = -(|\sigma+\rho_\frakk|^2-|\sigma|^2)\id_{\calH[\sigma]}, $$
	where $|\cdot|$ denotes the norm on $\frakt^*$ induced from the same invariant inner product on $\frakk$ that was used for the construction of the Casimir element $\Delta_\frakk$. It follows that the square of the Hilbert--Schmidt norm of the embedding $\calH^N\hookrightarrow\calH$ is given by
	$$ \sum_{\sigma\in\widehat{K}}\frac{\dim\calH[\sigma]}{\sum_{j=0}^N(|\sigma+\rho_\frakk|^2-|\sigma|^2)^{2j}} \leq C\sum_{\sigma\in\widehat{K}} (1+|\sigma|)^{-4N}\dim\calH[\sigma] $$
	for some $C=C_N>0$, where $\rho_\frakk\in i\frakt^*$ is half the sum of all positive roots. By a result of Harish--Chandra~\cite[Theorem 4]{HC54} combined with the Weyl Dimension Formula, there exist $C',C''>0$ such that
	$$ \dim\calH[\sigma] \leq C'(\dim\sigma)^2 \leq C''(1+|\sigma|)^{\dim\frakk-\dim\frakt} \qquad \mbox{for all }\sigma\in\widehat{K}. $$
	Hence, the square of the Hilbert--Schmidt norm of the embedding $\calH^N\hookrightarrow\calH$ can be bounded by a multiple of
	$$ \sum_{\sigma\in\widehat{K}} (1+|\sigma|)^{\dim\frakk-\dim\frakt-4N}. $$
	Summation is over the highest weight lattice in $i\frakt^*$, hence the sum is finite if and only if $\dim\frakk-\dim\frakt-4N<-\dim\frakt$.
\end{proof}

\begin{remark}
	Lemma~\ref{lem:SobolevEmbedding} essentially shows that $\calH^\infty$ is a nuclear Fr\'{e}chet space, which is well-known by the work of Harish-Chandra. Using this, most of the statements in Theorem~\ref{thm:Main} and Corollary~\ref{cor:Main} are also contained in \cite[Chapter 3.3]{Li18}.
\end{remark}

\section{Proof of the main results}

Combining Lemma~\ref{lem:SobolevEmbedding} with Theorem~\ref{thm:GKThm} shows that the restriction of $T$ to the Sobolev completion $\calH_\pi^N$ is pointwise defined for $N$ sufficiently large. Composing with the continuous linear embedding $\calH_\pi^\infty\hookrightarrow\calH_\pi^N$ and using Lemma~\ref{lem:DenseImageAndEquivariance}~\eqref{lem:DenseImageAndEquivariance2} shows Theorem~\ref{thm:Main}.

Now let $A_{\pi,\tau}^\infty:\calH_\pi^\infty\to\calH_\tau\otimes\calM_{\pi,\tau}$, $\tau\in\widehat{H}$, be as in Theorem~\ref{thm:Main}. For an orthonormal basis $(w_\alpha)_{\alpha=1,\ldots,m(\pi,\tau)}$ of $\calM_{\pi,\tau}$, $m(\pi,\tau)=\dim\calM_{\pi,\tau}\in\NN\cup\{\infty\}$, write
$$ A_{\pi,\tau}^\infty(v) = \sum_\alpha A_{\pi,\tau,\alpha}^\infty(v)\otimes w_\alpha \qquad (v\in\calH^\infty) $$
with $A_{\pi,\tau,\alpha}\in\Hom_H(\calH_\pi^\infty,\calH_\tau)$. If $A_{\pi,\tau}^\infty$ has dense image, then the operators $A_{\pi,\tau,\alpha}^\infty$ have to be linearly independent, so $m(\pi,\tau)=\dim\calM_{\pi,\tau}\leq\dim\Hom_H(\calH_\pi^\infty,\calH_\tau)$. Lemma~\ref{lem:DenseImageAndEquivariance}~\eqref{lem:DenseImageAndEquivariance1} implies that this is the case for almost every $\tau$, so Corollary~\ref{cor:Main} follows.


\providecommand{\bysame}{\leavevmode\hbox to3em{\hrulefill}\thinspace}
\providecommand{\href}[2]{#2}

\end{document}